\documentclass[12pt,fleqn]{article}   
\usepackage{latexsym}
\usepackage{amsmath}
\usepackage{amssymb}
\usepackage{epsfig}
\usepackage{amsfonts}

\usepackage{url}

\begin{document}
\newtheorem{definition}{Definition}[section]
\newtheorem{theorem}[definition]{Theorem}
\newtheorem{lemma}[definition]{Lemma}
\newtheorem{proposition}[definition]{Proposition}
\newtheorem{examples}[definition]{Examples}
\newtheorem{corollary}[definition]{Corollary}
\def\square{\Box}
\newtheorem{remark}[definition]{Remark}
\newtheorem{remarks}[definition]{Remarks}
\newtheorem{exercise}[definition]{Exercise}
\newtheorem{example}[definition]{Example}
\newtheorem{observation}[definition]{Observation}
\newtheorem{observations}[definition]{Observations}
\newtheorem{algorithm}[definition]{Algorithm}
\newtheorem{criterion}[definition]{Criterion}
\newtheorem{algcrit}[definition]{Algorithm and criterion}

\newenvironment{prf}[1]{\trivlist
\item[\hskip \labelsep{\it
#1.\hspace*{.3em}}]}{~\hspace{\fill}~$\square$\endtrivlist}
\newenvironment{proof}{\begin{prf}{Proof}}{\end{prf}}

\title{Effective descent for differential operators}
\author{
Elie {Compoint}\thanks{
D\'epartement de math\'ematiques,
Universit\'e de Lille I,
59655 , Villeneuve d'Ascq Cedex, France.
{\tt compoint@math.univ-lille1.fr} },\
Marius {van der Put}\thanks{
University of Groningen, Department of Mathematics
P.O. Box 407, 9700 AK Groningen, The Netherlands.
{\tt mvdput@math.rug.nl} },\
Jacques-Arthur {Weil}\thanks{ XLIM, D\'epartement de Math\'ematiques et Informatique,
       Universit\'e de Limoges,
       123 avenue Albert Thomas, 87060 Limoges { Cedex}, France.
       {\tt jacques-arthur.weil@unilim.fr} } \\
        \small{ Dedicated to the memory of Jerry Kovacic.}  }

\date{September 2009}

\maketitle

\begin{abstract} A theorem of N.~Katz \cite{Ka} p.45, states that an irreducible differential operator $L$
over a suitable differential field $k$, which has an isotypical decomposition over the algebraic closure
of $k$, is a tensor product $L=M\otimes _k N$ of an absolutely irreducible operator $M$ over $k$ and an irreducible operator $N$ over $k$ having a finite differential Galois group. Using the existence of the tensor decomposition $L=M\otimes N$, an algorithm is given in \cite{C-W}, which computes an absolutely irreducible factor $F$ of $L$ over a finite extension of $k$. Here, an algorithmic approach to finding
$M$ and $N$ is given, based on the knowledge of $F$. This involves a subtle descent problem for differential operators which can be solved for explicit differential fields $k$ which are $C_1$-fields.
\end{abstract}

\section{Introduction}
$C$ denotes an algebraically closed field of characteristic 0 and the differential field $k$ is a finite
extension of $(C(z), \partial =\frac{d}{dz})$. The algebraic closure of $k$ will be written as $\overline{k}$.
Let $L\in k[\partial ]$ be a (monic) differential operator. The operator $L$ is called irreducible if it does not factor over $k$ and absolutely irreducible if it does not factor over $\overline{k}$. Here we are interested
in the following special situation:\\

\noindent 
{\it  $L$ is irreducible and $L$ factors over $\overline{k}$ as a product $F_1\dots F_s$
of $s>1$ equivalent  (monic) absolutely irreducible operators.} \\

There are algorithms for factoring $L$ over $k$, i.e., as element of $k[\partial ]$ (\cite{H1,H2,vdP-S}).
Algorithms for finding factors of order $1$ in $\overline{k}[\partial ]$ are proposed in \cite{S-U, H-R-U-W, vdP-S}.
An algorithm for finding factors of arbitrary order in $\overline{k}[\partial ]$ is given in \cite{C-W}.

According to N.~Katz \cite{Ka}, Proposition 2.7.2,  p.45,  the assumption on $L$ implies that $L$ is a tensor product $M\otimes N$ of monic 
operators in $k[\partial ]$ such that $M$ is absolutely irreducible and the irreducible operator
$N$ has a finite differential Galois group (or equivalently all its solutions are algebraic over $k$). 
We will present a quick, down-to-earth proof of this in terms of differential modules over $k$. Further we note
that the converse statement is immediate because $N$ decomposes over $\overline{k}$ as a ``direct sum''
(least common left multiple)
of  operators $\partial -\frac{f'}{f}$ with $f\in \overline{k}^*$.

The absolute factorization algorithm in \cite{C-W} uses the existence of this tensor decomposition to ensure its correctness but,
in full generality, the problem of computing $M$ and $N$ is yet left open in there. For $M$ or $N$ of small order,
methods for detecting and computing $M$ and $N$ are given in \cite{C-W,N-vdP} and particularly in 
\cite{H3} where additional references can be found. We will illustrate this (and propose another method) at the 
end of the paper.

The problem which we address to produce the operators $M,N\in k[\partial ]$ by some decision procedure for 
$M$ and $N$ of arbitrary order. It follows
from $L=M\otimes N$ that $F_1\in \overline{k}[\partial ]$  is equivalent to $M$ (i.e., $F_1$ descends
to $k$). Let $K_1\supset k$ be the smallest Galois extension such that 
$F_1\in K_1[\partial ]$ (or equivalently 
$K_1$ is the field extension of $k$ generated by all the coefficients of all $F_i$). One might think that the equivalence between $F_1$ and $M$, seen as elements of $K_1[\partial ]$, takes place over $K_1$. However, in general, a (finite) extension $K'\supset K_1$ is needed for this equivalence. The title of this paper refers to this descent problem\footnote{Techniques for {\em arithmetic} descent were proposed in \cite{H-P}, where the case of a differential
field $k$ with non-algebraically closed constant field is handled}.  \\

Our method for finding $K'$ is as follows. First the smallest extension
$K\supset K_1$ is computed which guarantees that the factors $F_1,\dots ,F_s$
are equivalent over the field $K$. Using these equivalences,
  a certain 2-cocyle $c$ for $Gal(K/k)$ with values in $C^*$, i.e. the obstruction for the descent of $F_1$, is computed. Since $k$ is a $C_1$-field, the 2-cocycle $c$ becomes trivial over a finite (cyclic) computable extension $K'\supset K$~: we give a construction of $K'$ in section 3.1.2, particularly part (4) of remark 3.4 which produces first order operators having the same obstruction to descent and for which the problem can be solved. 
Finally, once $K'$ is found, the computation of $M,N$ is easily completed. \\

In the sequel we will use differential modules because these are more natural for the problem. A translation
in terms of differential operators is presented at the moment that actual algorithms are involved since the
latter are frequently phrased in terms of differential operators.

The following notation is used. The trivial 1-dimensional differential module over a
field $K$ is $Ke$ with $\partial e=0$. This module will be denoted by ${\bf 1}$
or ${\bf 1}_K$.  For a differential module $A$ over $K$ of dimension $a$ one writes $\det A$
for the 1-dimensional module $\Lambda ^aA$.

\section{A version of Katz' theorem}

In the proof we will use the following notion of {\it twist of a differential module}.\\
Let $Gal(K/k)$ be the Galois group of any Galois extension $K$ of $k$ (finite or infinite). Let $A$ be a differential module over $K$ and $\sigma \in Gal(K/k)$. The twist $^\sigma A$ is equal to $A$ as additive group, has the same 
$\partial$ as $A$, but its structure as $K$-vector space is given by 
$\lambda *a:=(\sigma ^{-1}\lambda )\cdot a$ for $\lambda \in K,\ a\in A$.

The elements $\sigma \in Gal(K/k)$ act in a natural way on $K[\partial ]$ by the
formula $\sigma (\sum _n a_n\partial ^n)=\sum _n\sigma (a_n )\partial ^n$.
If one presents $A$ as $K[\partial ]/K[\partial ]F$ (with $F$ monic), then 
$^\sigma A=K[\partial ]/K[\partial ]\sigma (F)$.

An isomorphism  $\phi (\sigma ):\ ^\sigma A\rightarrow A$ can also be interpreted as a $C$-linear bijection $\Phi (\sigma ):A\rightarrow A$, commuting with $\partial$ and such that $\Phi(\sigma )(\lambda a)=\sigma (\lambda )\cdot \Phi (a)$.  In other words $\Phi(\sigma )$ is $\sigma$-linear isomorphism.

\begin{proposition} Let $L$ be a differential module over $k$. Suppose that:\\
{\rm (1)}. The field of constants $C$ of $k$ is algebraically closed, $k$ has characteristic zero and $k$ is a $C_1$-field. \\
{\rm (2)}. $L$ is irreducible and $\overline{L}:=\overline{k}\otimes _kL$ decomposes
as a direct sum $\oplus _{i=1}^s A_i $ of isomorphic irreducible differential modules over
$\overline{k}$.

Then there are modules $M,N$ over $k$ such that $L\cong M\otimes _kN$, $M$ is absolutely irreducible
and the irreducible module $N$ has a finite differential Galois group.
The pair $(M,N)$ is unique up to a change $(M\otimes _kD,N\otimes _k D^{-1})$ where $D$ has dimension 1
and $D^{\otimes t}$ is the trivial module for some $t\geq 1$.
\end{proposition}
\begin{proof} Write $A=A_1$ and let $Gal$ denote the Galois group of $\overline{k}/k$. For any 
$\sigma \in Gal$, the twisted module $^\sigma A $ is a submodule of $^\sigma (\overline{k}\otimes _kL)$.
As the latter module is isomorphic to $\overline{k}\otimes _kL$, there is a
$\sigma$-linear  isomorphism $\Phi (\sigma ):  A \rightarrow A$. 

This induces a 2-cocycle $c$ for $Gal$ with values in $C^*$, defined
by $\Phi (\sigma \tau )=c(\sigma ,\tau )\cdot \Phi (\sigma)\cdot \Phi(\tau )$.  Since $k$ is a $C_1$-field, the 2-cocycle is trivial (\cite{Se}, II-9, \S 3.2).
After multiplying the isomorphisms  $\{\Phi (\sigma )\}$ by suitable elements in $C^*$, one obtains  descent data $\{\Phi (\sigma )|\ \sigma \in Gal\}$ satisfying the descent condition  $\Phi (\sigma \tau )=\Phi (\sigma)\cdot \Phi(\tau )$ for all $\sigma ,\tau \in Gal(\overline{k}/k)$.  Define $M:=\{a\in A\ |\  \Phi (\sigma ) a=a\mbox{ for all }\sigma \in Gal\}$. It is easily verified that $M$ is a differential module over $k$ and that the canonical morphism $\overline{M}:=\overline{k}\otimes _kM\rightarrow A$ is an isomorphism.

Consider now ${\rm Hom}_\partial (\overline{M},\overline{L})$. This is a vector
space, isomorphic to $C^s$ and provided with an action of $Gal$. Then 
$\overline{k}\otimes _C{\rm Hom}_\partial (\overline{M},\overline{L})$ is 
a trivial differential module over $\overline{k}$ provided with an action of $Gal$. It is a submodule of the differential module   ${\rm Hom}(\overline{M},\overline{L})$. Taking invariants  under $Gal$ one obtains a differential module 
\[N:=(\overline{k}\otimes _C{\rm Hom}_\partial (\overline{M},\overline{L}))^{Gal}
\mbox{ over }k \mbox{ which is a submodule  of } {\rm Hom}(M,L).\] 
The canonical morphism
$\overline{k}\otimes N\rightarrow  \overline{k}\otimes _C{\rm Hom}_\partial (\overline{M},\overline{L})$ is an isomorphism and thus the differential Galois group of $N$ is finite.

The canonical morphism of differential modules $M\otimes _k {\rm Hom}(M,L)\rightarrow L$, namely $m\otimes \ell \mapsto \ell (m)$, can be restricted to a morphism $ f: M\otimes _kN\rightarrow L$. By construction the induced morphism
 $\overline{M}\otimes _{\overline{k}}\overline{N}\rightarrow \overline{L}$ is an isomorphism and thus so is $f$. \\

The descent data  $\{ \Phi (\sigma )|\ \sigma \in Gal \}$ for $A$ are not unique. They can be changed into $\{h(\sigma ) \Phi (\sigma )|\ \sigma \in Gal\}$ where $h:Gal\rightarrow C^*$ is any continuous homomorphism. We note that the image of $h$ is a subgroup $\mu _t$ of the $t$-th roots of unity for some $t$. Consider the trivial differential module $\overline{k}e$ with $\partial e=0$ and with $Gal$ action given by $\sigma e=h(\sigma )e$ for all $\sigma \in Gal$. By taking the invariants under $Gal$ one obtains a 1-dimensional module $D$ over $k$ such that the canonical morphism $\overline{k}\otimes D\rightarrow \overline{k}e$ is an isomorphism and respects the actions of $Gal$.  Further $D^{\otimes t}$ is the trivial differential module ${\bf 1}$. 

The differential module obtained with the new descent data can be seen to be  $M\otimes _kD$ and thus $N$ will be changed into  $N\otimes D^{-1}$. From this observation the last statement of the proposition follows.  \end{proof} 

\begin{remarks} $\ ${\rm 
The non unicity of the pair $(M,N)$ can be restricted by the condition that $\det N={\bf 1}$. Then only a change $(M\otimes D,N\otimes D^{-1})$ is possible with $D^{\otimes s}={\bf 1}$.
}\end{remarks}
\section{Algorithmic approach}

First we make the  relation between differential modules and differential operators explicit. Let $A$ be a differential module and $a\in A$ a cyclic vector. One associates
to this the monic differential operator $op(A,a)\in k[\partial ]$ of minimal degree satisfying
$op(A,a)a=0$. The morphism $k[\partial ]\rightarrow A$, which maps $1$ to $a$,
induces an isomorphism $k[\partial ]/k[\partial ]op(A,a)\rightarrow A$.

Let $B\subset A$ be a submodule. This yields a factorization 
$op(A,a)=\mathcal{L}\mathcal{R}$ with $\mathcal{R}\in k[\partial ]$ is the monic operator of minimal degree such that $\mathcal{R}a\in B$. One observes that 
$\mathcal{R}=op(A/B,a+B)$.

Further $\mathcal{L}\in k[\partial ]$ is the operator of minimal degree satisfying 
$\mathcal{L}b=0$, where $b:=\mathcal{R}a$.  Clearly $b$ is a cyclic vector for $B$ and $\mathcal{L}=op(B,b)$.

Moreover, any factorization $op(A,a)=\mathcal{L}\mathcal{R}$ with monic $\mathcal{L},\mathcal{R}$ corresponds in this way to a unique submodule $B\subset A$,
namely $B=k[\partial ]\mathcal{R}a$.

Let $k'$ be an algebraic extension of $k$. Then the above bijection extends to a bijection
between the (monic) factorizations of $op(A,a)$ in $k'[\partial ]$ and the submodules of
$k'\otimes _kA$.\\

\noindent 
{\it As before,  $L$ denotes an irreducible differential module over $k$ such that $\overline{k}\otimes L$ is a direct sum of $s>1$ copies of an absolutely irreducible differential module}.\\

Choose $\ell \in L, \ell \neq 0$. Since $L$ is irreducible, $\ell$ is a cyclic vector. {\it We  assume the
knowledge of a factorization $op(L,\ell )=F.R$ with monic $F,R\in \overline{k}[\partial ]$ and $F$ absolutely irreducible}, given by \cite{C-W}.
Using this information we will describe the computation leading to a tensor product decomposition $L=M\otimes N$. \\

\subsection{The special case $\dim M=1$}
Assume that the irreducible $L$ is equal to $M\otimes _kN$ with 
$\dim M=1,\ \ \dim N=s>1$ and  $\overline{k}\otimes _kN$ is trivial.
Thus the Picard-Vessiot extension $K^+$ of $N$ is a finite extension of $k$ and
can be considered as a subfield of $\overline{k}$. The (covariant) solution space $V$ of
$N$ is equal to $\ker (\partial ,K^+\otimes _kN)$. The differential Galois group 
$G^+=Gal(K^+/k)$ acts on $V$ and there is a canonical isomorphism 
$K^+\otimes _CV\rightarrow K^+\otimes _kN$. Moreover, $\overline{M}:=\overline{k}\otimes _kM$ is not a trivial module (equivalently, the differential Galois group of $M$ is infinite and then equal to the multiplicative group $\mathbb{G}_m$).

There is a trivial way to produce a decomposition
$L=M\otimes _kN$. Indeed, write $op(L,\ell)=(\partial ^s+a_{s-1}\partial ^{s-1}+\cdots +a_0)$. Then the tensor product decomposition
$op(L,\ell )=(\partial +\frac{a_{s-1}}{s})\otimes (\partial ^{s}+b_{s-2}\partial ^{s-2}+\cdots +b_0)$,  for suitable elements $b_i\in k$,
solves already the problem, since (as one easily sees) all the solutions of $\partial ^{s-1}+b_{s-2}\partial ^{s-2}+\cdots +b_0$ are in $\overline{k}$. 
However, the aim of this subsection is to describe in this easy situation an algorithm for obtaining  the pair $(M,N)$,  up to a change $(M\otimes D, D^{-1}\otimes N)$, which 
works with small modifications for the general case. \\

After fixing a non zero element $\ell \in L$, the module is represented by the monic operator $op(L,\ell )$. {\it The first step} is to produce the smallest subfield $K\subset \overline{k}$ (containing $k$) such that $op(L,\ell )$ decomposes in $K[\partial ]$ as a product $F_1\cdots F_s$ of (monic) {\it equivalent} operators of degree 1.\\

The (monic) left hand factors $F=\partial +u\in \overline{k}[\partial ]$ of $op(L,\ell )$ 
correspond to the 1-dimensional submodules of 
\[\overline{k}\otimes L=M\otimes _k(\overline{k}\otimes _kN)=M\otimes _k
(\overline{k}\otimes _CV)\]
and these are the $M\otimes _k(\overline{k}\otimes _CW)$ where $W$ runs in the
set of the 1-dimensional  subspaces of $V$. The same can be done with $\overline{k}$
replaced by $K^+$. Therefore $u\in K^+$ and $K_0:=k(u)\subset K^+$. More precisely,
let $St(W)\subset G^+$ be the stabilizer of $W$. This is the subgroup of $G^+$ leaving
$F$ invariant and thus  $K_0= (K^+)^{St(W)}\subset K^+$.\\

{\it  Now we suppose that a (monic) left hand factor $F=\partial +u$
of $op(L,\ell )$ is known and explain how to obtain $K$ from this}.
 Let $K_1\subset \overline{k}$ be the normal closure of $K_0$. Then
$K_1\otimes _kL$ contains a 1-dimensional submodule that we will call again
$D$. The submodule $\sum _{\sigma \in Gal(K_1/k)}\sigma (D)$ of
$K_1\otimes _kL$ is invariant under the action of $Gal(K_1/k)$. Since $L$
is irreducible,  $\sum _{\sigma \in Gal(K_1/k)}\sigma (D)=K_1\otimes_k L$
and it follows that $K_1\otimes_k L$ is a direct sum of 1-dimensional submodules.
As a consequence, $op(L,\ell )$ factors as $F_1\cdots F_s$ in $K_1[\partial ]$.\\

A priori, the factors $F_i=\partial +u_i\in K_1[\partial ]$ need not be equivalent.
For $i<j$ we consider a non zero element $f_{ij}\in \overline{k}$ satisfying
$\frac{f_{ij}'}{f_{ij}}=u_i-u_j$. Put $K=K_1(\{f_{ij}\})$. Then $K\supset k$ is the smallest field such that $op(L,\ell )$ factors as $F_1\cdots F_s\in K[\partial ]$
where the monic degree one factors $F_i$ are equivalent.\\

Clearly $K\subset K^+$.  From the condition that  $K$ is minimal such that
$K\otimes _kN$ is a direct sum of isomorphic 1-dimensional submodules and
the irreduciblity of $N$, it follows easily that the center $Z$ of 
$G^+\subset {\rm GL}(V)$ is the finite cyclic group 
$(C^*\cdot id_V)\cap G^+$ and that $K=(K^+)^Z$.

\begin{remarks}{\rm (1) The field $K$ and the above algorithm for $K$ do not change
if $L$ is replaced by $D\otimes _k L$, where $D$ is a 1-dimensional module
satisfying $D^{\otimes t}={\bf 1}$ for some $t\geq 1$.\\ 
(2) The fields $K_0$ and $K_1$ depend on the given left hand factor of degree
1 of $op(L,\ell )$. We illustrate this by an example where the differential
Galois group $G^+\subset {\rm GL}(C^3)$ is generated by the matrices
$\left(\begin{array}{ccc} 0&0&1\\ 1&0&0\\ 0&1&0\end{array}\right)$ and 
$\left(\begin{array}{ccc} a&0&0\\ 0&b&0\\ 0&0&c\end{array}\right)$ with
$a^6=b^6=c^6=1$. If this left hand factor corresponds to $Ce_1$ (or 
$Ce_2$ or $Ce_3$), then its stabilizer is  the subgroup of the diagonal  
matrices in $G^+$. Otherwise it is just the center $Z$. In the first case
$K_1\neq K$ and in the last one $K_1=K$. }\hfill $\square$ \end{remarks}

\subsubsection{ The 2-cocycle $c$ and descent fields}
 Now we have arrived at the situation where $K\otimes _kL$ is a direct sum of $s$ copies of a  known 1-dimensional differential module $D=M\otimes _kE$ over $K$, where $E$ is an, a priori, unknown  1-dimensional submodule of $K\otimes _kN$. We want to produce a field extension $K'\supset K$ such that $K'\otimes _KD$ descends to $k$.

 Let $D$ correspond to the operator $\partial -u$. Then, for $\sigma \in Gal(K/k)$,  the operator $\partial -\sigma (u)$ corresponds to  $^\sigma D$. There are
elements $f_\sigma \in K^*$ such that $\sigma (u)-u=\frac{f_\sigma '}{f_\sigma}$.
One obtains a 2-cocycle $c$ for $Gal(K/k)$ with values in  $C^*$ by
$f_{\sigma \tau}=c(\sigma ,\tau )\cdot f_\sigma \cdot \ ^\sigma f_\tau$ for
all $\sigma ,\tau \in Gal(K/k)$. The class of the 2-cocycle $c$ in $H^2(Gal(K/k),C^*)$
 is the {\it obstruction for the descent} of  $D$ (or, equivalently, for the descent of  $E$).

  Indeed, if $D$ descends to $k$, then one can represent $D$ by $\partial -u$ with $u\in k$.
 On the other hand, if the class of $c$ is trivial, then after changing the $\{f_\sigma \}$
 one has $f_{\sigma \tau}= f_\sigma \cdot \ ^\sigma f_\tau$. By Hilbert 90, there exists
 $F\in K^*$ with $f_\sigma =\frac{\sigma F}{F}$ for all $\sigma \in Gal(K/k)$. Then
 $\partial -u$ is equivalent to $\partial -u+\frac{F'}{F}$. Further $u-\frac{F'}{F}$ lies in $k$,
 since it is invariant under $Gal(K/k)$.  

 \bigskip

 The exact sequence 
 $1\rightarrow Z\rightarrow G^+\stackrel{pr}{\rightarrow} Gal(K/k)\rightarrow 1$ 
 induces a 2-cocycle with values in $Z$, in the following way. Let
  $\phi :Gal(K/k)\rightarrow G^+$ 
 be a section, i.e., $pr\circ \phi (g)=g$ for all $g\in Gal(K/k)$. Then $d$, defined by
 $\phi (g_1g_2)=d(g_1,g_2)\phi (g_1)\phi (g_2)$, is a 2-cocycle with values in $Z$.
 The class of $d$ in $H^2(Gal(K/k),Z)$  is independent of the choice of the section $\phi$.
 As before, $Z$ is identified with a subgroup of $C^*$. Thus $d$ induces an element of
 $H^2(Gal(K/k),C^*)$. We note that the homomorphism $H^2(Gal(K/k),Z)\rightarrow
 H^2(Gal(K/k),C^*)$ is, in general, not injective. 

 \begin{lemma}  The cocycles $c$ and $d$ have the same image in $H^2(Gal(K/k),C^*)$.
 This image does not depend of the choice of $N$.  Let $s=\dim N$. Then the image of $c^s$ in $H^2(Gal(K/k),C^*)$  is trivial.
 \end{lemma}
  \begin{proof}
   Let $D$, as before, be given by $\partial -u$ with $u\in K$.   Write $u=\frac{F'}{F}$ with 
   $F\in (K^+)^*$. Let $\phi :G\rightarrow G^+$ be a section. Then 
$\sigma (u)-u=\frac{\phi (\sigma)F'}{\phi (\sigma) F}-\frac{F'}{F}=\frac{f_\sigma '}{f_\sigma}$
where $f_\sigma :=\frac{\phi (\sigma )F}{F}$. One easily sees that $f_\sigma $ is invariant
under $Z$ and thus $f_\sigma \in K^*$.
The equality $f_{\sigma \tau }=c(\sigma ,\tau )f_\sigma \ ^\sigma f_\tau$ implies
$\phi(\sigma \tau )F=c(\sigma ,\tau )\phi (\sigma) \phi (\tau)F$. Hence 
$d(\sigma ,\tau)=c(\sigma ,\tau)$. 

Replacing $N$ by $(\partial -v)\otimes _kN$ with $v\in k$, induces the change 
of $\partial -u$ into $\partial -u-v$. Since $\sigma (u+v)-(u+v)=\sigma (u)-u$, the element 
$c$ and its image in $H^2(Gal(K/k),C^*)$ are unchanged.

Suppose that $N$ is chosen such that $\det N={\bf 1}$. The  cocycle $d$ has values
in $\mu _s$, since $\dim N=s$. Thus the image of $d^s$ in $H^2(Gal(K/k),C^*)$ is 
trivial and the same holds for $c^s$.    \end{proof}

\begin{definition} {\rm  Let $(K,c)$ be a Galois extension $K/k$ and
$c$  a 2-cocycle for $Gal(K/k)$ with values in $C^*$ and such that $c^s$ is a trivial
2-cocycle.  A {\it descent field} for $(K,c)$
is a Galois extension $K'\supset k$ containing $K$, such that the induced 2-cocycle 
$c'$ for $Gal(K'/k)$, defined by $c'(g_1,g_2)=c(pr g_1,pr g_2)$, yields a trivial element
in $H^2(Gal(K'/k),C^*)$. Here $pr: Gal(K'/k)\rightarrow Gal(K/k)$ denotes the natural map. 
}\hfill $\square$ \end{definition}

The assumption that $k$ is a $C_1$-field implies the {\it existence of a descent field for every
pair $(K,c)$}. Indeed, by \cite{Se}, II \S 3, the cohomological dimension
of $Gal(\overline{k}/k)$ is 1 and therefore $H^2(Gal(\overline{k}/k),\mu_\infty )=1$,
where $\mu_\infty$ denotes the torsion subgroup of $C^*$.
 One has $H^2(Gal(K/k),C^*)=H^2(Gal(K/k),\mu _\infty)$ and
$\underset{\rightarrow }{\lim}\ H^2(Gal(K'/k),\mu_\infty )=H^2(Gal(\overline{k}/k),\mu_\infty )=\{1\}$,
where the direct limit is taken over all Galois extensions $K'\supset k$, containing $K$.   
Hence for every class $\overline{c}\in H^2(Gal(K/k),C^*)$ there exists a 
Galois extension $K'\supset k$, containing $K$, such that the image of $\overline{c}$
in $H^2(Gal(K'/k),C^*)$ is $1$.\\
 {\it Our contribution is now to produce a descent field for $(K,c)$ by some algorithm}.

\subsubsection{A decision procedure constructing a descent field for $(K,c)$}

Since $k$ is a $C_1$-field,  $H^2(Gal(K/k),K^*)$ is trivial (\cite{Se}, II-9, \S 3.2) and this implies the existence of elements 
$\{f_\sigma |\ \sigma \in Gal(K/k)\}\subset K^*$ satisfying $f_{\sigma \tau}=c(\sigma ,\tau )\cdot f_\sigma \cdot \ ^\sigma f_\tau$ for all $\sigma ,\tau \in Gal(K/k)$.

{\it Suppose that $c$ is given in the form
$f_{\sigma \tau }=c(\sigma ,\tau )f_\sigma \ ^\sigma f_\tau$ with $\{f_\sigma \}\subset K^*$}.
(This is trivially true for the present case $\dim M=1$. 
For the general case, Remarks 3.4 part (4) describes a decision procedure 
producing suitable $\{f_\sigma \}$, a key step in the construction). 
{\it Then the following algorithm produces a descent field}. \\

One has
$\frac{f_{\sigma \tau}'}{f_{\sigma \tau }}=\frac{f_\sigma '}{f_\sigma }+
\sigma (\frac{f_\tau '}{f_\tau})$ and since $H^1(Gal(K/k),K)=0$  there is
an element $v\in K$ such that $\frac{f_\sigma '}{f_\sigma }=\sigma (v)-v$
for all $\sigma \in Gal(K/k)$;  explicitly 
\[v=\frac{-1}{[K:k]}\sum _{\tau \in Gal(K/k)}  \frac{f_\tau '}{f_\tau}.\]

One observes that $c$ is the obstruction for descent of the operator $\partial -v$.\\ 
Further $-mv=\frac{G'}{G}$ with $G=\prod _{\tau \in Gal(K/k)}f_\tau$ and $m=[K:k]$. Hence the field $K(\sqrt[ m ]{G})$ contains the Picard-Vessiot field of $\partial -v$\
and is  a descent field.

A less brutal way to compute a descent field is as follows.
Since the cocycle $c^s$ is trivial, there are computable elements 
$\{d(\sigma )|\ \sigma \in Gal(K/k)\}\subset C^*$
satisfying $d(\sigma \tau )=c(\sigma ,\tau )^sd(\sigma )d(\tau )$. The elements
$\{\frac{f_\sigma ^s}{d(\sigma )}\}$ form a 1-cocycle. Since $H^1(Gal(K/k),K^*)=\{1\}$,
one can effectively compute $F\in K^*$ such that 
$\frac{f_\sigma ^s}{d(\sigma )}=\frac{\sigma F}{F}$ for all $\sigma \in Gal(K/k)$
(see \cite{Se2}, Chapitre X, \S 1, Prop. 2). 
One observes that $v-\frac{1}{s}\frac{F'}{F}\in k$ since it is invariant under $Gal(K/k)$. 
The field extension $K'=K(\sqrt[s]{F})$  has the property that 
$(\partial -v)$ is equivalent to $\partial -v+\frac{1}{s}\frac{F'}{F}$ over $K'$.
Hence $K'$ is a descent field. \hfill $\square$ \\

\begin{remarks}{\rm $\ $\\
(1)  We note that the above algorithm proves, by considering 1-dimensional differential modules over $K$, the existence of a descent field only using that  $H^2(Gal(K/k),K^*)=\{1\}$ (see Lemma 4.1 for the general statement).\\
(2). Instead of assuming that the 2-cocycle $c^s$ is trivial, we
may consider a class $\overline{c}\in H^2(Gal(K/k),\mu _\infty )$, where 
$\mu_\infty \subset C^*$ denotes, as before, the group of the roots of unity. Any finite group $G$ occurs as
some $Gal(K/k)$. Therefore the group $H^2(Gal(K/k),\mu _\infty)$ is in general not
trivial and the descent problem, i.e., finding an extension $K'\supset K$ such that
the image of $\overline{c}$ in $H^2(Gal(K'/k),\mu _\infty )$ is 1, is non trivial.
However for a cyclic $Gal(K/k)$ one has $H^2(Gal(K/k),\mu _\infty )=\{1\}$
(\cite{Se2}, VIII, \S 4). In particular, non trivial examples for the descent problem tend to be complicated.\\
(3). We ignore how to compute or characterize all minimal descent fields
for a given pair $(K,c)$.\\
(4).   Computing  elements $f_\sigma \in K^*$ satisfying 
$f_{\sigma \tau}=c(\sigma ,\tau )\cdot f_\sigma \cdot \ ^\sigma f_\tau$ appears to be far from trivial. 
A possible 
method, which uses explicitly 
the $C_1$-property of $k$, is the following. 
Starting with the 2-cocyle $c$, there is a well known construction (see \cite{G-S}) of an algebra
$A=\oplus _{\sigma \in G}K[\sigma ]$, where $G=Gal(K/k)$, of dimension 
$m=\#G=[K:k]$ over $K$, defined by the rules:
\[ [\sigma ]\cdot \lambda =\sigma (\lambda)\cdot [\sigma ]\mbox{ for } \lambda \in K,\
\sigma \in G,\]
 \[ [\sigma _1\sigma _2]=c(\sigma _1,\sigma _2)[\sigma _1]\cdot [\sigma _2] .\]
Then $A$ is a central simple algebra with center $k$. Since $k$ is a $C_1$-field
there exists an isomorphism $I:A\rightarrow Matr(m,k)$. The latter algebra can be identified
(by standard Galois theory) with the group algebra $K[G]=\oplus K\cdot \sigma$. 
Suppose that we knew already elements $f_{\sigma}\in K^*$ satisfying
$f_{\sigma \tau }=c(\sigma ,\tau )\cdot f_\sigma \cdot \ ^\sigma f_\tau$. 
 Then $I_{0} : A \mapsto K[G]$, given by $I_{0}(\sum \lambda_{\sigma}[\sigma])=
 	\sum \lambda_{\sigma} f_{\sigma}.\sigma$ is an isomorphism and is also $K$-linear.
By the Skolem-Noether theorem, any isomorphism $I$ of $k$-algebras has the form
$I(\sum \lambda_{\sigma}[\sigma])=x^{-1}\{\sum \lambda_{\sigma} f_{\sigma}.\sigma\}x$,
where $x$ is an invertible element of $Matr(m, k)$. Our aim is to compute a 
$K$-linear isomorphism and in that case $x$ commutes with $K$ and therefore 
belongs to $K^*$ . Now $I$ has the form 
$I(\sum \lambda_{\sigma}[\sigma])= \sum \lambda_{\sigma} f_{\sigma}.\frac{\sigma(x)}{x}.\sigma$
 and thus
any $K$-linear isomorphism $A\rightarrow K[G]$ has the form $[\sigma]\mapsto g_{\sigma} · \sigma$ for suitable 
elements $g_{\sigma}\in K^{*}$ . It follows that 
	$g_{\sigma \tau }=c(\sigma ,\tau )\cdot g_\sigma \cdot \ ^\sigma g_\tau$. 
%

The computation of the isomorphism $I$ uses the reduced norm $Norm$ of $A$
(see e.g. \cite{F,Pi} or \cite{R} section 4, which adapts to $C_{1}$ fields).   With respect to a basis of $A$ over $k$, the reduced norm is a homogeneous form of degree $m$ in $m^2$ variables. Again the $C_1$ property of $k$ asserts that there are
non trivial solutions $a\in A,\ a\neq 0$ for $Norm(a)=0$.
An explicit calculation of such $a$ is possible (but rather expensive). Applying this several times one obtains the isomorphism $I:A\rightarrow Matr(m,k)$.\\
\begin{example}
Consider the case where
$K\supset k$ has degree 2 and $c$ is a 2-cocycle for $G=\{1,\sigma \}$ with values in
$K^*$.

One easily sees that the 2-cocyle can be given by $c(1,1)=c(1,\sigma )=c(\sigma ,1)=1$
and $c(\sigma ,\sigma )=\alpha ^{-1} \in k^*$. The $K$-linear morphism
$\phi :A:=K[1]\oplus K[\sigma ]\rightarrow K[G]=K\oplus K\sigma$ should have the form
$\phi ([1])=1,\ \phi ([\sigma ])=f\sigma$ with $f\in K^*$. We have to find $f$. Now the
condition is $\alpha =f\sigma (f)$.   Write $K=k\oplus  kw$ with $w^2\in k^*$ and write
$f=a+bw$. Then we have to solve $a^2-b^2w^2=\alpha$. 
Consider the equation
$X_1^2-X_2^2w^2-X_3^2\alpha =0$. By the $C_1$-property of $k$ there is a solution
$(x_1,x_2,x_3)\neq 0$.  Now $x_3\neq 0$, since $w^2\in k^*$ is not a square. Then we
can normalize to $x_3=1$ and the problem is solved.
\end{example}
(5) For $\dim N=2$ and $M$ of any dimension we will give in the Appendix 
an easier algorithm for descent fields, not using the 2-cocycle $c$ explicitly (and recall the former algorithms of
\cite{H3,C-W,N-vdP}).\\
(6) Let again a Galois extension $k\subset K$ with Galois group $G$ and 2-cocycle
$\overline{c}\in H^2(G,C^*)$ be given. The 2-cocycle class has finite order (dividing
$s$) and corresponds to a short exact sequence
$1\rightarrow Z\rightarrow G^+\rightarrow G\rightarrow 1$, where $Z$ is a finite cyclic group, lying in the center of $G^+$. Suppose that the Galois extension $k\subset K^+$
with group $G^+$ is such that $(K^+)^Z=K$. Then the image of $\overline{c}$ in
$H^2(G^+,C^*)$ is trivial and the descent condition holds for the field $K^+$. The
$C_1$-property of the field $k$ guarantees the existence of $K^+$, however there seems
to be no explicit algorithm, based on the $C_1$-property, producing an $K^+$.  
Examples 4.4 are based on this remark.\\ 
(7). Finally, we note that $H^2(Gal(K/k),\mu _s)=1$ if $g.c.d.([K:k],s)=1$. In that
case there is no field extension needed for the descent. }\hfill $\square$ \end{remarks}

\subsection {Description of the algorithm for the general case}
We will search for a decomposition $L=M\otimes N$ with $\det N=1$.
The module $\overline{k}\otimes L$ can be written as
$M\otimes _k(\overline{k}\otimes _kN)=M\otimes _k(\overline{k}\otimes _CV)$
where $V$ is the solution space of $N$. The absolutely irreducible left hand factors
$F$ of $op(L,\ell )$ correspond to the 1-dimensional subspaces $W$ of $V$.

We suppose that at least one $F$ is given. Let $K_0\supset k$ denote the
field extension generated by the coefficients of $F$. Let $K_1$ be the normal 
closure of $K_0$.  For each $\sigma \in Gal(K_1/k)$ one considers the absolutely irreducible left hand factor $\sigma (F)$. This factor is over $\overline{k}$ equivalent  
to $F$. We have to compute the field extension of $K_1$ needed for this equivalence.  

An algorithm in terms of differential modules (which easily translates in terms of
differential operators) is based upon the following lemma.

\begin{lemma} Let $A$ be an irreducible differential module over $\overline{k}$.
Then the  differential module ${\rm Hom}(A,A)$ over $\overline{k}$ has only one 
1-dimensional submodule, namely $\overline{k}\cdot id_A$.
\end{lemma}
\begin{proof}  It is possible to prove this by using \cite{N-vdP} and irreducible representations of semi-simple Lie algebras. 

However, a more down-to-earth proof is the following. Let $V$ be the solution space of $A$. This is a  $C$-vector space of dimension equal to 
$a:=\dim _{\overline{k}}A$, equipped with a faithful irreducible action of the differential Galois group $G$ of $A$.
The group $G$ is connected since $\overline{k}$ is algebraically closed.
A 1-dimensional submodule of ${\rm Hom}(A,A)$ corresponds to a 1-dimensional subspace $Cf$ of ${\rm Hom}_C(V,V)$, invariant under the action of $G$. There is a homomorphism $c:G\rightarrow C^*$ such that $gfg^{-1}=c(g)\cdot f$ holds 
for all $g\in G$. The kernel of $f$ is $G$-invariant and is $\{0\}$ since the representation is irreducible. The action of $G$ on 
${\rm Hom}(\Lambda ^aV,\Lambda ^aV)$ is trivial. In particular, the isomorphism
$\Lambda ^a(f) :\Lambda ^aV\rightarrow \Lambda ^aV$ is invariant under
$G$. Also $g(\Lambda ^a(f))g^{-1}=c(g)^a\cdot \Lambda ^a(f)$ and 
$c(g)^a=1$. Since $G$ is connected, $c(g)=1$ for all $g\in G$. Thus $f$ is $G$-invariant and is a multiple of $id_V$ since the representation is irreducible. 
\end{proof}

\begin{corollary} The  differential module 
\[T(\sigma ):={\rm Hom}(K_1[\partial ]/K_1[\partial ]\sigma (F),K_1[\partial ]/K_1[\partial ]F)\]
over $K_1$ has a single 1-dimensional submodule $A(\sigma )$. Moreover, the 
Picard--Vessiot field of $A(\sigma )$ is a finite cyclic extension $K_1'\subset \overline{k}$ of $K_1$. \end{corollary}
\begin{proof} $S:=\overline{k}\otimes _{K_1}T(\sigma )$ is isomorphic to 
${\rm Hom}(\overline{M},\overline{M})$, where 
$\overline{M}:=\overline{k}\otimes _k M$.  
By Lemma 3.5, $S$ has a unique 1-dimensional submodule, say, $B$. By uniqueness, $B$ is invariant under the action of $Gal(\overline{k}/K_1)$ on $S$ and has therefore the form $\overline{k}\otimes _{K_1}A(\sigma )$ for some submodule $A(\sigma )$ of $T(\sigma )$. The uniqueness of $A(\sigma )$ is clear.

Further $\overline{k}\otimes A(\sigma )$ is isomorphic to the trivial differential module $\overline{k}\cdot id_{\overline{M}}$. Thus the Picard--Vessiot field $K_1'$
is a finite extension of $K_1$ and this extension is cyclic since $A(\sigma )$ has
dimension 1. \end{proof}

By factorization the 1-dimensional submodule $A(\sigma )$ can be obtained.
The Picard--Vessiot field of $A(\sigma )$ is a finite cyclic extension $K_1'\subset \overline{k}$ of $K_1$.
  Then $\ker (\partial ,K_1'\otimes T(\sigma) )$ has dimension 1 over $C$ and
a generator $\phi(\sigma)$ of this kernel is an isomorphism
$\phi (\sigma ):K_1'[\partial ]/K_1'[\partial ]\sigma (F)\rightarrow K_1'[\partial ]/K_1'[\partial ]F$. \\

The field $K\subset \overline{k}$ is the compositum of the Picard--Vessiot 
fields  of all $A(\sigma )$. We note that $K$ is the field called ``{\it stabilisateur}''
in \cite{C-W}.
The isomorphisms 
$\phi (\sigma ):K[\partial ]/K[\partial ]\sigma (F)\rightarrow K[\partial ]/K[\partial ]F$
are now also known, they are $K$-rational solutions of the  modules 
$K\otimes _{K_1}A(\sigma )$.  The 2-cocycle $c$ for $Gal(K/k)$ with  values in $C^*$ 
has the property that $c^s$ is trivial by the assumption that $\det N={\bf 1}$.
Then, as in Subsection 3.1, one can construct a cyclic extension $K'\supset K$ such that
the module $K'[\partial ]/K'[\partial ]F$ descends to $k$. The result is called $M$.

The module $N$ is obtained by computing the unique irreducible direct summand of  $M^*\otimes _k L$ having dimension $s$. Indeed, this direct summand of 
$M^*\otimes _k L=M^*\otimes_kM\otimes _kN={\rm Hom}(M,M)\otimes _kN$
is $(k\cdot id_M)\otimes _kN\cong N$.

\subsubsection{An exemple for the construction of a descent field}
Consider the irreducible operator
$$ L  = \partial^{4}+{\frac { \left( -4+8 z \right) }{{z}^{2}-1}}\partial^3
+{\frac { \left( 53 {z}^{2}-40 z-1 \right) }{4 \left( {z}^{2}-1 \right) ^{2}}}\partial^2
 +{\frac { \left( 5 {z}^{3}-{z}^{2}-13 z-3 \right)}{ 2\left( {z}^{2}-1 \right) ^{3}}} \partial
+{\frac {61 {z}^{2}+64 z+67}{ 16\left( {z}^{2}-1 \right) ^{4}}}.$$
The algorithm of \cite{C-W} produces the following absolutely irreducible right-hand factor
$$L_{1}=\partial^{2} +{\frac { \left( 3\,{\it u}+4\,{z}^{2}-26\,z+24\right)}{2 \left( {z}^{2}-1 \right)  \left( 4\,z-5 \right) }}\partial
+{\frac { \left( -3\,z-6 \right) {\it u}+45\,{z}^{2}-40\,z-13}{4 \left( {z}^{2}-1 \right) ^{2} \left( 4\,z-5 \right) }},$$
where $u^2=z^2-1$. It is isomorphic to its conjugate $L_{2}$ over
$K=k(\Phi)$ with $\Phi^4-2z\Phi^2+1=0$ (or $\Phi=\sqrt{z-u}$). Explicitely,
there exists $S\in K[\partial]$ such that $L_{2}.R=S.L_{1}$ with 
$$R={\frac { \left( 1-2\,z+2\,{\it u} \right) }{\Phi( 4\,z- 5) } }
\left( 2\, \left( {z}^{2}-1 \right) \partial-z-2\right),$$
i.e. $R$ maps a solution of $L_{1}$ to a solution of $L_{2}$. 
\\
Let $G=Gal(K/k)$, acting via
$\sigma_{1}(\Phi)=\Phi$, $\sigma_{2}(\Phi)=-\Phi$, $\sigma_{3}(\Phi)=1/\Phi$,  $\sigma_{4}(\Phi)=-1/\Phi$.
The $2$-cocycle $c$ is given by 
$c(\sigma_{2},\sigma_{3})=c(\sigma_{2},\sigma_{4})=c(\sigma_{4},\sigma_{3})=c(\sigma_{4},\sigma_{4})=-1$
(and $c(\sigma_{i},\sigma_{j})=1$ otherwise).\\
Remark 3.4 part (4) (see also Example 4.3 in the Appendix) produce the elements $f_{\sigma}$ which are, respectively,  $1,1,\Phi,\Phi$~;
hence the construction from section 3.2.1 shows that $K'=K(\sqrt{\Phi})$ is a descent field.
An anihilating operator for $\sqrt{\Phi}$ is
$$N={\partial}^{2}+{\frac {z}{{z}^{2}-1}}\partial-1/16\, \left( {z}^{2}-1 \right) ^{-1}$$
(one could find it by writing $K'$ as a $k[\partial]$-module and decomposing it).
\\
Decomposing $L\otimes N^\star$ (over $k$), we then obtain
$$M={\partial}^{2}-{\frac {2}{z-1}}\partial +{\frac {35\,z+37}{16  \left( z+1 \right)  \left( z-1 \right) ^{2}}}$$
and $L_{1}$ is isomorphic over $K'$ to $M$. At the end of the appendix, we give alternative (easier) methods
to handle such small order examples.

\section{Appendix}

The following lemma makes the relation between 2-cocycles and descent for
1-dimensional differential modules more explicit. We present some examples
and present an algorithm producing descent fields for the case $\dim N=2$.

 \begin{lemma} Let $k$ be a differential field having the properties: the field of
 constants $C$ is algebraically closed and has characteristic 0; $k$ is a $C_1$-field.

  Let $K/k$ be a Galois extension (finite or infinite). The collection $H(K)$ of the (isomorphy classes of the) 1-dimensional differential modules $A$ over $K$, satisfying 
  $^\sigma A\cong A$ for all $\sigma \in Gal(K/k)$, forms a group with respect to the 
  operation tensor product. Let $h(K)\subset H(K)$ denote the subgroup consisting
   of the modules of the form $K\otimes _kB$, where $B$ is a 1-dimensional differential module over $k$.

    There is a canonical isomorphism $H(K)/h(K)\rightarrow H^2(Gal(K/k),C^*)$.
   \end{lemma}
 \begin{proof} The first statement is obvious. Let the differential module $Ke$ with
 $\partial e=ue$ lie in $H(K)$. Then for any $\sigma \in Gal(K/k)$ there is an
 element $f_\sigma \in K^*$ such that $\sigma (u)-u=\frac{f_\sigma '}{f_\sigma }$.
 Define the 2-cocycle $c$ by $f_{\sigma \tau }=c(\sigma ,\tau )\cdot f_\sigma \cdot \
 ^\sigma f_\tau$ for all $\sigma ,\tau \in Gal(K/k)$.  Replacing the $f_\sigma$ by
 $d(\sigma )f_\sigma $, with $d(\sigma )\in C^*$, changes the 2-cocycle into an equivalent one. Tensoring $Ke$ with an element of $h(K)$ changes $u$ into
 $u+v$ with $v\in k$ and this does not effect the $f_\sigma$. Thus the above construction defines a homomorphism $H(K)/h(K)\rightarrow H^2(Gal(K/k),C^*)$.

 This homomorphism is injective since the triviality of the 2-cocycle class $\overline{c}$
 implies that $f_{\sigma \tau }=  f_\sigma \cdot \  ^\sigma f_\tau$. Since
 $H^1(Gal(K/k),K^*)=\{1\}$ there is an $F\in K^*$ with $f_\sigma =\frac{\sigma F}{F}$
 for all $\sigma$. Thus $\sigma (u)-u=\sigma (\frac{F'}{F})-\frac{F'}{F}$ and 
 $\overline{u}:=u-\frac{F'}{F}$ is invariant under $Gal(K/k)$ and belongs to $k$.
 Now $Ke=K\cdot \overline{e}$ with $\overline{e}:=F^{-1}e$ and $\partial \overline{e}=\overline{u} \overline{e}$. Thus $Ke$ belongs to $h(K)$.

 The homomorphism is surjective.   Indeed,  consider a 2-cocycle $c$ for $Gal(K/k)$
 with values in $C^*$. Since $H^2(Gal(K/k),K^*)=\{1\}$ there are elements
 $f_\sigma \in K^*$ such that $f_{\sigma \tau }=c(\sigma ,\tau )\cdot f_\sigma \cdot \
 ^\sigma f_\tau$ for all $\sigma ,\tau \in Gal(K/k)$. Then 
 $\frac{f'_{\sigma \tau }}{f_{\sigma \tau }}=\frac{f_\sigma '}{f_\sigma}+\sigma (
 \frac{f'_\tau}{f_\tau})$ and since $H^1(Gal(K/k),K)=\{0\}$ there is an element
 $u\in K$ such that $\sigma (u)-u=\frac{f'_\sigma }{f_\sigma }$ for all 
 $\sigma \in Gal(K/k)$.  Thus the class of $c$  is the image of the module $Ke$ with 
 $\partial e=ue$, belonging to $H(K)$.   \end{proof} 

\begin{example} {\rm Let $k=C(z)$ and $K=C(t)$ with $t^2=z$. Let $\sigma$ be the non trivial element in $Gal(K/k)$. \\
The module $Ke$ with $\partial e=ue$ belongs to $h(K)$ if and only if 
\[u=w+\frac{1}{2t}\sum _{\alpha \neq 0}\frac{d_\alpha \sqrt{\alpha}}{z-\alpha }
\mbox{ with all }d_\alpha \in \mathbb{Z}  \mbox{ and } w\in k .\]
(We note that $\sqrt{\alpha}$ denotes an arbitrary choice of a square root for
$\alpha\in C^*$).
This follows from the computation: if $K^*\ni F=t^{n_0}\prod _{\beta \neq 0}
(t-\beta )^{n_\beta}$, then 
\[\frac{F'}{F}=\frac{n_0}{2t^2}+\sum _{\beta ^2\neq 0}
\frac{n_\beta +n_{-\beta}}{2(t^2-\beta ^2)}+ 
\frac{1}{2t}\sum _{\beta ^2\neq 0}\frac{n_\beta \beta -n_{-\beta }\beta}{t^2-\beta ^2}.\]

Consider now $Ke$ with $\partial e=ue$, belonging to $H(K)$. Write 
$u=a+\frac{1}{2t}b$ with $a,b\in k$. By assumption $\frac{-1}{t}b=\sigma (u)-u$
has the form $\frac{G'}{G}$ for some $G\in K^*$. Write 
$G=t^{m_0}\prod _{\beta \neq 0} (t-\beta )^{m_\beta}$, then, using the above formula,
one finds that $m_0=0$ and $m_{-\beta}=-m_{\beta}$. Thus
$b=\sum _{\alpha \neq 0}\frac{m_{\sqrt{\alpha}}\sqrt{\alpha}}{z-\alpha }$, where
$\alpha =\beta ^2$ (and some choice of $\sqrt{\alpha}$ is made). According
to the above result, one has  that $Ke$ lies in $h(K)$. Therefore 
$H(K)/h(K)=\{0\}$.  This is in accordance with $H^2(Gal(K/k),C^*)=\{1\}$ (since
$Gal(K/k)$ is cyclic).    }\hfill  $\square$ \end{example}

\begin{example} {\rm The group $D_2\cong (\mathbb{Z}/2\mathbb{Z})^2$
has the property that $H^2(D_2,C^*)$ contains an element of order 2. In order to
obtain this element for an obstruction to descent we consider the following 
differential fields
\[ k=C(z)\subset K=C(t)\subset K'=C(s)\mbox{ with } z=t^2+t^{-2},\ t=s^2 .\]
The group $Gal(K/k)=\{1,a,b,ab\} \cong D_2$ where the elements $a,\ b$ are given by
$a(t)=-t$ and $b(t)=t^{-1}$. One considers the differential module $Ke$ with 
$\partial e=ue$ and $u=\frac{1}{4(t^2-t^{-2})}=\frac{s'}{s}$. One observes that
$a(u)-u=0$ and $b(u)-u=\frac{(t^{-1})'}{t^{-1}}$. Thus we may take 
$f_1=1,\ f_a=1,\ f_b=t^{-1},\ f_{ab}=t^{-1}$. The corresponding 2-cocycle $c$ has values
in $\{\pm 1\}$ and  $f_{ab}=c(a,b)\cdot f_a\cdot \ ^af_b$ holds with $c(a,b)=-1$.  It is
easily verified that $\overline{c}\in H^2(D_2,C^*)$ is not trivial. The module 
$K'\otimes _Ke$ is trivial since $u=\frac{s'}{s}$ and therefore descends to $k$.
In particular $K'$ is a descent field for $(K,c)$.}\hfill $\square$ \end{example}

We note that for a finite group $G$, acting trivially on $C^*$, the cohomology
group $H^2(G,C^*)$ is called the Schur multiplier of $G$. This group is well
studied, see \cite{Su}. 

\begin{examples}{\rm  A construction of many examples of the type
$L=M\otimes _kN$ under consideration, involving a non trivial descent problem,
is the following.

Suppose that $G^+$ is given as a finite irreducible subgroup of ${\rm GL}(V)$ where $\dim _CV=n>1$  and that the center $Z$ of $G^+$ is non trivial.  Further,
assume that a Galois extension $K^+\supset k=C(z)$  with Galois group $G^+$ is given. Then
$K:=(K^+)^Z$ is a Galois extension of $k$ with group $G:=G^+/Z$.

Consider the  differential  module $K^+\otimes _CV$ over $K^+$, defined by
$\partial ( f\otimes v)=f'\otimes v$ for $ f\in K^+,\ v\in V$. This is a trivial differential module. The action of $G^+$ on $K^+\otimes _CV$ is defined by 
$\sigma (f\otimes v)= \sigma (f)\otimes \sigma (v)$.  This action commutes with
$\partial$. 

Define $N:=(K^+\otimes _CV)^{G^+}$. This is an irreducible differential module over $k$ with Picard--Vessiot field $K^+$. The subfield $K$ is the smallest field such that $K\otimes _kN$ is a direct sum of isomorphic copies of a 
1-dimensional differential module $D$ over $K$. In particular $^\sigma D\cong D$
for all $\sigma \in Gal(K/k)$. The 2-cocycle attached to $D$ is non trivial if there
is no subgroup $H\subset G^+$ mapping bijectively to $G$. More precisely,
$K^+$ is a smallest field over which the 2-cocycle becomes trivial if and only
if no proper subgroup $H$ of $G^+$ maps surjectively to $G$. 

Take now any absolutely irreducible differential module $M$ of dimension 
$m>1$ over $k$. Then $L:=M\otimes _kN$ has the required properties.
For $\dim N=2$ there is a rich choice of examples and there are similar explicit  cases 
for $n=3$, see \cite{vdP-U}.  }\hfill $\square$ \end{examples}


\begin{example} Algorithms for the descent field for the case $\dim N=2$.\\ {\rm
Let $L=M\otimes N$ be given with $M$ absolutely irreducible and an irreducible
$N$ with $\dim N=2,\ \det N={\bf 1}$ and finite differential Galois group $G^+$.
For this case, methods for finding $N$ (hence the descent field) and $M$
are proposed, e.g, in \cite{H3,C-W,N-vdP} (and references therein).

Below is another nice method, adapted to our case.
We have $G^+
\in \{D_k^{SL_2},\ A_4^{SL_2},\  S_4^{SL_2},\ A_5^{SL_2}\}\subset  {\rm SL}(2,C)$ 
with center $Z=\{\pm {1\ 0\choose 0\ 1 }\}$  and  there is no proper subgroup of 
$G^+$ mapping onto  $G:=G^+/Z\in \{D_k,\ A_4,\ S_4,\ A_5\}$. 
According to Lemma 3.2,
the 2-cocycle class $\overline{c}=\overline{d}\in H^2(G,C^*)$ is non trivial. The method
of Section 3 provides the field $K$ with $Gal(K/k)=G$, from the data $op(L,\ell )$ 
and an absolutely irreducible monic left hand factor $F$ of $op(L,\ell )$.

The (unknown) group $G^+$ is a subgroup of ${\rm SL}(V)$ with $\dim _CV=2$.
Write $W=sym^2V$ and let $S\in sym^2 (W)$ be a generator of the kernel of
$sym^2(W)\rightarrow sym^4(V)$. Then $S$ is a non degenerate symmetric form of degree two. The homorphism $\psi :{\rm SL}(V)\rightarrow {\rm SL}(W)$, defined by 
$A\mapsto A\underset{s}{\otimes} A$, has kernel $\{\pm {1\ 0\choose 0\ 1}\}$ and its image  is $\{B\in {\rm SL}(W)|\  S \mbox{ is invariant under  } B\}$.   One observes that 
$\psi (G^+)=G$. Conversely, for any subgroup $G\subset {\rm SL}(W)$ preserving
the form $S$, one has that $\psi ^{-1}(G)=G^+$.

Let $a\in K$ be a general element, then the orbit $Ga$ is a basis of $K/k$ and
the $C$-vector space with basis $Ga$ is the regular representation of $G$. This
vector space contains an irreducible representation $W$ of $G$ of dimension three.
In each of the cases for $G$, there exists a (unique) non degenerated symmetric form
$S$ on $W$ which is invariant under $G$. 

The unique monic differential operator $T_3\in K[\partial ]$ of degree 3 which is 0 on $W$ belongs to $k[\partial ]$ because $W$ is invariant under $G$. This operator (or the
corresponding differential module) is equivalent to the  second symmetric power of an operator
$T_2\in k[\partial ]$ (which can be found using \cite{H3}). Let $\tilde{K}$ denote the Picard-Vessiot field of $T_2$. Then
$[\tilde{K}:K]=2$.  Let $V\subset \tilde{K}$ denote the space of solutions of $T_2$.
Then $W=\{v_1v_2|\ v_2,v_2\in V\}=sym ^2V$. The differential Galois group of
$T_2$ is $\psi ^{-1}(G)$ and thus isomorphic to $G^+$. Hence $\tilde{K}$ is a descent
field for $(K,d)$ and then also for $(K,c)$. Using this descent field one computes
$M$ and $N$.\\

Yet another observation (though less practical) is that J.J.~Kovacic's fundamental algorithm for order 2 equations,
\cite{Ko}, could also be applied to $L=M\otimes N$. For example, the symmetric power
$sym^{m+1}(L)$ contains an irreducible factor over $k$, which is projectively isomorphic to $M$, 
for $m=2,4,6,12$ for the cases
$G^+=D_k^{SL_2},A_4^{SL_2},S_4^{SL_2},A_5^{SL_2}$.  More refined
factorisation patterns may be established for each of these cases.}\end{example}

\begin{example} The referee's example. {\rm The operator
\[ L_4=\partial ^4+\frac{6z}{z^2-1}\partial ^3+
\frac{1971 z^2-947}{288(z^2-1)^2}\partial ^2+\frac{27z}{32 (z^2-1)^2}\partial +
\frac{9}{4096 (z^2-1)^2}\mbox{ has the}\] 
\[ \mbox{absolutely irreducible {\em right} hand factor }L_2= \partial ^2+\frac{3z-\alpha +1}{6(z^2-1)}\partial
+\frac{3}{64 (z^2-1)},\]   
where $\alpha$ is a root of $T^4+12(z-1)T^2-32(z-1)T-12(z-1)^2=0$.

There are the following methods:\\
(1). By \cite{C-W,H3}. $L_4 = M\otimes N$ with $\det M=\det N={\bf 1}$. The two factors of  
$\Lambda ^2 L_4 = sym^2 (M)\oplus sym^2(N)$ are easily computed and,
using \cite{H3}, one finds $M$ and $N$.\\
(2). By  \cite{N-vdP}, Theorem 6.2. One computes $F\in sym^2( L_4)$ with 
$\partial F=0,\ F\neq 0$ and  a 2-dimensional isotropic subspace for $F$. From 
the last part of the proof of \cite{N-vdP}, Theorem 6.2 one reads off $M$ and $N$. \\
(3). Example 4.5 works here as follows. $K_0=k(\alpha)$ and its normal closure
$K_1$ has Galois group $A_4$ and $K=K_1$.  The $C$-vector space $W$ spanned by $A_4\alpha$ has dimension 3. The operator 
$T_3=\partial ^3+a_2\partial ^2+a_1\partial +a_0\in k[\partial ]$ with solution space
$W$  is determined by the equation $T_3(\alpha )=0$. This yields
\[T_3=\partial ^3+\frac{3z-1}{(z+1)(z-1)}\partial ^2+
\frac{27z+5}{36(z+1)^2(z-1)^2}\partial -\frac{9z+23}{36(z+1)^2(z-1)^3}.\]
\[\mbox{The operator } 
T_2=\partial ^2+\frac{3z-1}{3(z+1)(z-1)}\partial -\frac{3z-11}{48(z+1)(z-1)^2}\] 
 satisfies $sym^2(T_2)=T_3$ and  its Picard-Vessiot field (an extension of $k$ of degree $24$)
is a descent field. A minimum polynomial of an algebraic solution of $T_{2}$ is
$$P={Y}^{8}+ \frac{1}{3}\left( z-1 \right) {Y}^{4}+ \frac {4}{27}\left( z-1 \right) {Y}^{2}-{\frac {1}{108}}\, \left( z-1 \right) 
^{2}.$$
It factors over $k(\alpha)$ as 
$$  \left( {Y}^{2}+ \frac{\alpha}{6} \right)  \left( {Y}^{6}-\frac{\alpha}{6}\,{Y}^{4}
+ \frac{1}{3}\left( z-1+\frac{{\alpha}^{2}}{12} \right) {Y}^{2}-{\frac {{\alpha}^{3}}{216}}-\frac{1}{18} \left( z-1 \right) 
\alpha+{\frac {4}{27}}(z-1)\right) 
, $$
which illustrates the fact that $K^+$ is obtained from $K_{1}$ by adjunction of
a square root, here $\sqrt{-\frac{\alpha}{6}}$ (in fact, adjoining any solution of $T_{2}$ would do).
\\ Continuation of our method (decomposing $L_{4}\otimes T_{2}^*$ over $k$) then yields
$$M=\partial^{2}-\frac{1}{3}\,{\frac { \left( -1+3\,x \right)}{{x}^{2}-
1}}\partial +{\frac {1}{192}}\,{\frac {189\,{x}^{2}-96\,x+227}{  ({x}^{2}-1)^2   }}$$
and we may check that 
$$(M\otimes T_{2}).\left( (x^2-1)\partial \right) = \left( \left( {x}^{2}-1 \right) \partial+4\,x \right).L_{4}.$$
As solutions of both $M$ and $T_{2}$ can be expressed in terms of special functions (e.g using the
methods of van Hoeij), this allows to solve $L_{4}$ in terms of algebraic and special functions.

}

\end{example}

\begin{center}
{\bf Acknowledgments}
\end{center}
The authors are grateful to Mark van Hoeij and Lajos Ronyai and to the referee for interesting suggestions and 
 for providing Example 4.6 leading to an improvement of our paper.

   \end{document}